\documentclass{amsart}

\usepackage{amsmath,amssymb,amsthm,mathrsfs, dcpic, pictex,etex, graphicx}
\usepackage[margin=1in]{geometry}
\begin{document}
\fontsize{10.95}{14}\rm
\newcommand{\oper}[1]{\operatorname{#1}} 
\newcommand{\p}{\mathbb{P}^1}
\newcommand{\z}{\mathbb{Z}}
\newcommand{\xx}{\mathcal{X}}
\newcommand{\yy}{\mathcal{Y}}
\newcommand{\co}{\mathbb{C}}
\newcommand{\pp}{\mathscr{P}}
\newcommand{\f}{\mathbb{F}_p}
\newcommand{\fc}{\overline{\mathbb{F}_p}}
\newcommand{\iso}{\to^{\!\!\!\!\!\!\!\sim\,}}
\newcommand{\q}{\mathbb{Q}}
\newcommand{\tm}{((t^{\frac{1}{m}}))}
\newcommand{\zun}{\z_p^{\oper{un}}}
\newcommand{\qun}{\q_p^{\oper{un}}}
\newcommand{\kar}{\overline{K}}
\newcommand{\pa}{\mathfrak{p}}
\newcommand{\al}{\widetilde{\mathbb{Z}}^{\overline{\mathbb{Q}}}}
\newcommand{\lal}{\widetilde{\mathbb{Z}_p}^{\overline{\mathbb{Q}_p}}}

\title{The fields of definition of branched Galois covers of the projective line}

\author{Hilaf Hasson}
\date{January 19, 2013}
\begin{abstract}
In this paper I explore the structure of the fields of definition of Galois branched covers of the projective line over $\bar \q$. The first
main result
states that every mere cover model has a unique minimal field of definition where its automorphisms are defined, and goes on to describe
special properties of this field. One
corollary of this result is that for every $G$-Galois branched cover there is a field of definition which is Galois over its field of
moduli, with Galois group a subgroup of $\oper{Aut}(G)$. The second main theorem states that the field resulting by adjoining to the field
of moduli all of the roots of unity whose order divides some power of $|Z(G)|$ is a field of definition. By combining this result with
results from an earlier paper, I prove corollaries related to the Inverse Galois Problem. For example, it allows me to prove that for every
finite group $G$, there
is an extension of number
fields $\q\subset E\subset F$ such that $F/E$ is $G$-Galois, and $E/\q$
ramifies only over those primes that divide $|G|$. I.e., $G$ is realizable over a field that is
``close'' to $\q$.
\end{abstract}
\maketitle
\theoremstyle{plain}
\newtheorem{thm}{Theorem}[section]
\newtheorem{rmk}[thm]{Remark}
\newtheorem{qst}[thm]{Question}
\newtheorem{prp}[thm]{Proposition}
\newtheorem{hyp}[thm]{Hypotheses}
\newtheorem{crl}[thm]{Corollary}
\newtheorem{cnj}[thm]{Conjecture}
\newtheorem{stt}[thm]{Statement}
\newtheorem{lem}[thm]{Lemma}
\newtheorem{dff}[thm]{Definition}
\newtheorem{clm}[thm]{Claim}
\newtheorem{ntt}[thm]{Notation}
\newtheorem{example}[thm]{Example}
\section{Overview}\label{E:overview}
The Inverse Galois Problem asks whether every finite group $G$ is
realizable as a Galois group over $\q$ (or more generally over every number field $K$). Most attempts to solve
the Inverse Galois Problem over a number field $K$ have focused on trying to solve its geometric analogue, the Regular Inverse Galois
Problem. The Regular Inverse Galois Problem asks whether for every finite group $G$ there is a $G$-Galois branched cover of the projective
line over $\bar \q$ that is defined (together with its automorphisms) by polynomials with coefficients in $K$. It is well known that
for every finite group $G$ there is a $G$-Galois branched
cover of the projective line over $\bar \q$. (This is proven via transcendental methods; see Remark
\ref{E:riemann} for more details.) While most previous work has focused on the field of moduli (see Definition \ref{E:defoffom}) of such
covers, the focus of this paper is on the structure of their fields of definition.

In Section \ref{E:introduction} we provide an introduction to the definitions and concepts in this paper. In Section \ref{E:meresection} we
give a bijection between mere cover models and a group-theoretic object. (See Lemma \ref{E:modelssections}.) This allows us
to prove the first main theorem of this paper in Section \ref{E:minimalformodel}, namely Theorem \ref{E:tech}. This theorem states
that every mere cover model of a $G$-Galois branched cover has a unique minimal field where its automorphisms are defined, and this field of
definition has
special properties. This theorem has several noteworthy corollaries. Among them, it follows that for every $G$-Galois branched cover of
$\p_{\bar \q}$ there is a field of definition that is Galois over the field of moduli, with Galois group a subgroup of $\oper{Aut}(G)$.
(See Corollary \ref{E:thebigone}.) In particular, there is always a ``small'' field of definition over the field of moduli. Finally, in
Section
\ref{E:adjoin} we construct a special field of definition (infinite over the field of moduli) for every $G$-Galois branched cover, resulting
from adjoining certain elements to
the field of moduli. (See Theorem \ref{E:abelian}.) This, together with results from a previous paper (\cite{me}), allow us to prove
several corollaries (gathered in Corollary \ref{E:coolerigp}). For example, it allows us to prove that for every finite group $G$, there
is an extension of number
fields $\q\subset E\subset F$ such that $F/E$ is $G$-Galois, and $E/\q$
ramifies only over those primes that divide $|G|$. I.e., $G$ is realizable over a field that is
``close'' to $\q$.

This paper is based in large part on portions of the author's doctoral thesis, written at the University of Pennsylvania under the
supervision of David
Harbater.

\section{Introduction and Definitions} \label{E:introduction}
\begin{ntt}\rm
 Given an integral scheme $S$, we write $\kappa(S)$ for its function field.
\end{ntt}

\begin{dff} \label{E:defofgalois}\rm
Let $K$ be a field, and let $X_K$ and $Y_K$ be connected, normal, complete curves over $K$. We say that a map $X_K\rightarrow Y_K$ of
$K$-curves is a \it branched
cover \rm (or simply a cover) if the map is finite and generically \'etale. We say  that a
branched cover is \it Galois \rm if the induced extension of
function fields $\kappa(X_K)/\kappa(Y_K)$ is a Galois extension of fields. We
sometimes refer to branched covers as
\it mere covers\rm.

Let $G$ be a finite group. A \it $G$-Galois branched cover \rm is a branched cover
$X_K\rightarrow Y_K$ which is Galois, together with an isomorphism of $\oper{Gal}(\kappa(X_K)/\kappa(Y_K))$
with $G$.
\end{dff}
\begin{dff}\rm 
 Let $G$ be a finite group, and let $X_{\bar \q}\rightarrow Y_{\bar \q}$ be a $G$-Galois branched cover of curves over $\bar \q$. We say
that $K\subset \bar \q$ is a \it field of definition of $X_{\bar \q}\rightarrow Y_{\bar \q}$ as a mere cover \rm if it descends to a map of
$K$-curves $X_K\rightarrow Y_K$. (Any such $X_K\rightarrow Y_K$ is called a $K$-model of $X_{\bar \q}\rightarrow Y_{\bar \q}$.) We say
that $K$ is a \it field of definition as a $G$-Galois branched cover \rm if $X_{\bar \q}\rightarrow Y_{\bar \q}$ has a $K$-model that is
Galois.
\end{dff}
Let $G$ be a finite group. In this paper we will be interested in $G$-Galois branched covers $X_{\bar \q}\rightarrow\p_{\bar \q}$ of the
projective line. Such covers have a special importance in Galois Theory. Namely, if a number field $K$
is a field of definition of $X_{\bar
\q}\rightarrow\p_{\bar \q}$ as a $G$-Galois branched cover then Hilbert's Irreducibility Theorem (\cite{fj}, Chapter 11) implies that $G$ is
the Galois group of a
Galois field extension of $K$. In particular, if for every finite group $G$ there is a $G$-Galois branched cover $X_{\bar
\q}\rightarrow\p_{\bar \q}$ that descends to $\q$ (as a $G$-Galois branched cover) then the answer to the Inverse Galois Problem is
affirmative.
\begin{rmk}\label{E:riemann}\rm 
Let $a_1,...,a_r$ be closed points of $\p_{\co}$. Riemann's Existence Theorem (see \cite{sgaone}, expos\'e XII) states that every
topological covering space of $\p_{\co}\smallsetminus\{a_1,...,a_r\}$ is defined by polynomials. It follows that there is an
equivalence of
categories between $G$-Galois branched covers of $\p_{\co}\smallsetminus\{a_1,...,a_r\}$ that are \'etale and principal $G$-bundles of the
induced
topological space. Since the (topological) fundamental group of the Riemann Sphere punctured at $r$ points is free with $r-1$ generators, it
follows that it has a principal $G$-bundle for every finite group $G$ that is generated by $r-1$ elements. In particular it implies that for
every finite group $G$ there exists a $G$-Galois branched cover of $\p_{\co}$.
In fact, if we choose $a_1,...,a_r$ so that they come from closed points of $\p_{\bar \q}$ it follows from an argument of Grothendieck
that the cover descends to $\bar \q$.
Therefore, for every finite group $G$ there exists a $G$-Galois branched cover of $\p_{\bar \q}$. However, since the proof of Riemann's
Existence Theorem is not constructive, very little is known about the fields of definition of these covers.
\end{rmk}
Previous work on the structure of fields of definition of $G$-Galois
branched covers (resp. mere covers) has concentrated on the ``field of moduli''. The field of moduli is a field naturally
associated
to a $G$-Galois branched cover (resp. mere cover), and is the best candidate for the smallest field of definition (if one exists). 

\begin{dff}\label{E:defoffom}
 \rm 
 Let $G$ be a finite group, and let $X_{\bar \q}\rightarrow Y_{\bar \q}$ and $X'_{\bar \q}\rightarrow Y_{\bar \q}$ be $G$-Galois branched
covers of $Y_{\bar \q}$. We say that they are \it isomorphic as mere covers \rm if there exists an isomorphism $\eta$ that makes the
following commute:

$$
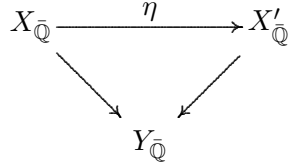
\begindc{\commdiag}[45]
\obj(0,1)[do]{$X_{\bar \q}$}
\obj(1,0)[dt]{$Y_{\bar \q}$}
\obj(2,1)[df]{$X'_{\bar \q}$}
\mor{do}{dt}{}
\mor{do}{df}{$\eta$}
\mor{df}{dt}{}
\enddc$$

If $\eta$ commutes with the given isomorphisms of $\oper{Gal}(\kappa(X_{\bar \q})/\kappa(Y_{\bar \q}))$ and $\oper{Gal}(\kappa(X'_{\bar
\q})/\kappa(Y_{\bar \q}))$ with $G$, we say that $X_{\bar \q}\rightarrow Y_{\bar \q}$ and $X'_{\bar \q}\rightarrow Y_{\bar \q}$ are \it
isomorphic as $G$-Galois 
branched covers\rm.
 
 Let $X_{\bar \q}\rightarrow Y_{\bar \q}$ be a $G$-Galois branched cover of curves 
over a field $\bar \q$. Let $K$ be a subfield of $\bar \q$. The \it field of moduli \rm of
$X_{\bar \q}\rightarrow Y_{\bar \q}$ as a $G$-Galois branched cover (resp. mere cover) relative to
$K$ is the subfield of $\bar \q$ fixed by those automorphisms of $\oper{Gal}(\bar \q/K)$
that take the $G$-Galois branched
cover (resp. mere cover) to an isomorphic copy of itself. We will use the convention that the field of moduli is always taken relative to
$\q$,
unless otherwise stated.
\end{dff}
Let $G$ be a finite group, and let $X_{\bar \q}\rightarrow \p_{\bar \q}$ be a $G$-Galois branched cover. It is clear that the field of
moduli of $X_{\bar \q}\rightarrow \p_{\bar \q}$ as a $G$-Galois branched cover (resp. mere cover) is contained in all of its
fields of definition as a $G$-Galois branched cover (resp. mere cover).

David Harbater and Kevin Coombes have proven in \cite{ch} that the field of moduli of $X_{\bar \q}\rightarrow \p_{\bar \q}$, considered as a
$G$-Galois branched cover (resp. mere cover) is in fact equal to the intersection of all of its fields of definition as a $G$-Galois
branched cover (resp. mere cover). Furthermore, the field of moduli of $X_{\bar \q}\rightarrow \p_{\bar \q}$ as a mere cover is a field
of definition as a
mere cover, and therefore the unique minimal field of definition as a mere cover.

It is important to note that the field of moduli
of a $G$-Galois branched cover of the projective line is not necessarily a field of definition as a $G$-Galois branched cover. In other
words, a $G$-Galois
branched cover may not have a \it unique \rm minimal field of definition. The obstruction for the field of moduli $M$ of a $G$-Galois
branched cover to being a field of definition (as a $G$-Galois branched cover) lies in $H^2(M, Z(G))$. (See \cite{belyi}, \cite{groupes}
and \cite{padiccovs}. The reader may also wish to consult \cite{descent}.) In particular,
if $G$ is centerless or if $M$ has cohomological dimension $1$ it follows that the field of moduli is a field of definition. In
\cite{wewfield}
Stefan Wewers has explored this obstruction in detail.

\section{Mere Cover Models and Sections}\label{E:meresection}
Let $G$ be a finite group, and let $X_{\bar
\q}\rightarrow Y_{\bar \q}$ a
$G$-Galois branched cover of normal complete curves over $\bar \q$. Let $L$ be a field of definition of $X_{\bar \q}\rightarrow Y_{\bar \q}$
as
a mere
cover, and let $Y_L$ be an $L$-model of $Y_{\bar \q}$. Let $\Omega$ be the set of mere cover
models $X_L\rightarrow
Y_L$ of $X_{\bar \q}
\rightarrow Y_{\bar \q}$ over $L$ that lie above $Y_L$. The goal of this section is to give a bijection between $\Omega$ and the set of
sections of some epimorphism of pro-finite groups. In order to do that, we require some notation.

We have following diagram of fields:

$$
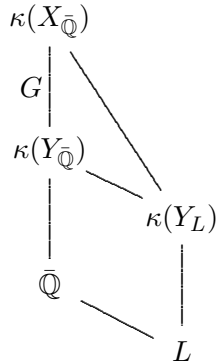
\begindc{\commdiag}[25]
\obj(0,3)[do]{$\kappa(Y_{\bar \q})$}
\obj(0,1)[dt]{$\bar \q$}
\obj(2,2)[df]{$\kappa(Y_L)$}
\obj(2,0)[zz]{$L$}
\obj(0,5)[za]{$\kappa(X_{\bar \q})$}
\mor{do}{dt}{}[\atright,\solidline]
\mor{do}{df}{}[\atright,\solidline]
\mor{df}{zz}{}[\atright,\solidline]
\mor{zz}{dt}{}[\atright,\solidline]
\mor{do}{za}{$G$}[\atleft,\solidline]
\mor{za}{df}{}[\atright,\solidline]
\enddc$$

Since we assumed $L$ is a field of definition as a mere cover, Lemma 2.4 in \cite{syb2} (see also \cite{matzat}) implies that
$\kappa(X_{\bar \q})$ is
Galois over $\kappa(Y_L)$. 

We have a short exact sequence:
$$1\rightarrow G \rightarrow \oper{Gal}(\kappa(X_{\bar \q})/\kappa(Y_L))\rightarrow \oper{Gal}(\kappa(Y_{\bar \q})/\kappa(Y_L))\rightarrow
1$$

Let $\oper{Gal}(L)$ denote the absolute Galois group $\oper{Gal}(\bar \q/L)$. Let $f:\oper{Gal}(\kappa(X_{\bar
\q})/\kappa(Y_L))\twoheadrightarrow \oper{Gal}(L)$ be the
composition of the quotient map $\oper{Gal}(\kappa(X_{\bar
\q})/\kappa(Y_L))\twoheadrightarrow \oper{Gal}(\kappa(Y_{\bar \q})/\kappa(Y_L))$ with the isomorphism $\oper{Gal}(\kappa(Y_{\bar
\q})/\kappa(Y_L))\iso \oper{Gal}(L)$. In other words, the map $f$ takes an automorphism $\sigma$ in $\oper{Gal}(\kappa(X_{\bar
\q})/\kappa(Y_L))$ to the restriction $\sigma|_{\bar \q}$ of $\sigma$ to $\bar \q$. We get the following short exact sequence.
$$1\rightarrow G \rightarrow \oper{Gal}(\kappa(X_{\bar \q})/\kappa(Y_L))\xrightarrow[]{f} \oper{Gal}(L)\rightarrow 1$$

Let $\oper{Sec}(f)$ denote the set of sections of $f$ in the category of pro-finite groups.

Let $X_L\rightarrow Y_L$ in $\Omega$ be a mere cover model of $X_{\bar \q}\rightarrow Y_{\bar \q}$. Note that $\kappa(X_{\bar \q})$ is
naturally isomorphic to the tensor product $\bar \q\otimes_L\kappa(X_L)$. We denote by $\omega_{X_L/Y_L}:\oper{Gal}(L)\rightarrow
\oper{Gal}(\kappa(X_{\bar \q})/\kappa(Y_L))$ the map taking $\sigma$ to $\sigma\otimes
\oper{id}_{\kappa(X_L)}$.

\begin{lem} \label{E:modelssections}

  In the above situation, the following
hold:
\begin{enumerate}
 \item Let $\alpha:\Omega\rightarrow\oper{Sec}(f)$ be the map taking a mere cover model $X_L\rightarrow Y_L$ to $\omega_{X_L/Y_L}$. Then
$\alpha$ is a bijection.
 \item Let $X_L\rightarrow
Y_L$ be a mere cover model of $X_{\bar \q}\rightarrow Y_{\bar \q}$. Then $X_L\rightarrow Y_L$ is Galois if and only if the image of
$w_{X_L/Y_L}$ commutes
with $G$.
 
\end{enumerate}

\end{lem}
\begin{proof}

In order to prove that $\alpha$ is onto, we first prove that for every section $s\in \oper{Sec}(f)$, the field $L$ is algebraically closed
in $\kappa(X_{\bar \q})^{s(\oper{Gal}(L))}$. It is straightforward to see that the
natural map $\oper{Ker}(f)\rightarrow\oper{Gal}(\kappa(X_{\bar \q})/\kappa(Y_L))/s(Gal(L))$ is a bijection of sets. Therefore the field
$\kappa(X_{\bar
\q})^{s(\oper{Gal}(L))}$ has
degree $|\oper{Ker}(f)|=|G|$ over $\kappa(Y_L)$. This implies that $\kappa(Y_{\bar \q})$ is linearly disjoint from $\kappa(X_{\bar
\q})^{s(\oper{Gal}(L))}$ over $\kappa(Y_L)$, and therefore $L$ is algebraically closed in $\kappa(X_{\bar
\q})^{s(\oper{Gal}(L))}$.

It follows from the above that there is a mere cover model
$X_{L,s}\rightarrow Y_L$ that induces the field extension $\kappa(X_{\bar \q})^{s(\oper{Gal}(L))}/\kappa(Y_L)$, and that
the field $\kappa(X_{\bar \q})$ is equal to the the compositum $\bar \q\cdot \kappa(X_{L,s})$. Let $\sigma$ be an element of
$\oper{Gal}(L)$. Since both $s(\sigma)$ and $w_{X_{L,s}/Y_L}(\sigma)$ restrict to $\sigma$ on $\bar \q$, and restrict to the trivial
automorphism on $\kappa(X_{L,s})$, it follows that $s(\sigma)$ is equal to $w_{X_{L,s}/Y_L}(\sigma)$. In
other words, $\alpha$ is onto.

In order to finish the proof of Claim (1) of this lemma, it remains to prove that $\alpha$ is injective. Let $X_L\rightarrow Y_L$ be
an element of $\Omega$. As we have seen above, the field extension\\ $\kappa(X_{\bar \q})^{w_{X_L/Y_L}(\oper{Gal}(L))}/\kappa(Y_L)$ has
degree $|G|$. It is
clear by the definition
of $w_{X_L/Y_L}$ that
$\kappa(X_L)$ is contained in $\kappa(X_{\bar \q})^{w_{X_L/Y_L}(\oper{Gal}(L))}$. Since $\kappa(X_L)$ also has degree $|G|$ over
$\kappa(Y_L)$ it follows that $[\kappa(X_{
\bar \q})^{s(\oper{Gal}(L))}:\kappa(X_L)]=1$, and they are equal. In other words you can recover the mere-cover model $X_L\rightarrow
Y_L$ from its induced section. This concludes
the proof of Claim (1) of the lemma.

It remains to prove Claim (2) of this lemma, i.e.\ that given a mere cover model $X_L\rightarrow Y_L$ the group
$w_{X_L/Y_L}(\oper{Gal}(L))$ commutes
with $G$ if and only if $X_L\rightarrow Y_L$ is Galois. As we have seen above $\kappa(X_L)$ is equal to $\kappa(X_{\bar
\q})^{w_{X_L/Y_L}(\oper{Gal}(L))}$. Therefore, by Galois Theory, the cover $X_L\rightarrow Y_L$ is Galois exactly when
$w_{X_L/Y_L}(\oper{Gal}(L))$ is normal in $\oper{Gal}(\kappa(X_{\bar \q})/\kappa(Y_L))$. Since $\oper{Gal}(\kappa(X_{\bar
\q})/\kappa(Y_L))$ is the semi-direct product of $G$ and $\oper{Gal}(\kappa(X_{\bar \q})/\kappa(Y_L))$, this is equivalent to
$w_{X_L/Y_L}(\oper{Gal}(L))$ commuting with $G$.
\end{proof}
\begin{rmk}\label{E:functorialconstruction}\rm
The construction of $\alpha$ is functorial in the following sense. Let $E$ be an overfield of $L$ that is contained in $\bar \q$, and let
$Y_E$ be $Y_L\times_LE$. Let $\Omega_E$ be the set of mere cover models $X_E\rightarrow Y_E$ of $X_{\bar \q}\rightarrow Y_{\bar \q}$ lying
above $Y_E$. Let $\alpha'$ be the bijection between $\Omega_E$ and $\oper{Sec}(g)$ where $g:\oper{Gal}(\kappa(X_{\bar
\q})/\kappa(Y_E))\twoheadrightarrow \oper{Gal}(E)$ is the composition of the quotient map $\oper{Gal}(\kappa(X_{\bar
\q})/\kappa(Y_E))\twoheadrightarrow \oper{Gal}(\kappa(Y_{\bar \q})/\kappa(Y_E))$ with the isomorphism $\oper{Gal}(\kappa(Y_{\bar
\q})/\kappa(Y_E))\iso \oper{Gal}(E)$. Then $\alpha'(X_L\times_LE\rightarrow Y_E)$ is the
restriction of $\alpha(X_L\rightarrow Y_L)$ to $\oper{Gal}(E)$.
\end{rmk}

\section{Minimal Fields of Definition of a Given Model}\label{E:minimalformodel}

The main theorem (Theorem \ref{E:tech}) of this section states that every mere cover model of a $G$-Galois branched
cover has a unique minimal field of definition that makes it Galois, and explores the special properties of this field. This result is
somewhat surprising, since it is well known that if you do not fix the mere cover model there may not be a unique minimal field of
definition for the automorphisms. (See Remark \ref{E:explanation} for further discussion.)

In order
to prove Theorem \ref{E:tech} we require a group-theoretic lemma (Lemma \ref{E:group}).
\begin{ntt}
\rm Let $g$ and $h$ be elements in a group $G$. We use the notation $^hg$ to mean the conjugation $hgh^{-1}$.
\end{ntt}
\begin{lem} \label{E:group}
 Let $J$ and $M$ be groups, and let $I$ be a semi-direct product
$J\rtimes M$. Let $N$
be $M\cap C_I(J)$, where $C_I(J)$ is the centralizer of $J$ in $I$. Then the following hold:
\begin{enumerate}
 \item $N$ is normal in $I$.
 \item Let $\gamma:M/N\rightarrow \oper{Aut}(J)$ be defined by taking $mN$
to the automorphism $j\mapsto \,^mj$. Then $\gamma$ is well defined, and injective.
 \item $I/N$ is isomorphic to  the semi-direct product $J\rtimes_{\gamma}
(M/N)$.
\end{enumerate}
\end{lem}
\begin{proof}
 Since $J$ is normal in $I$, it follows that so is $C_I(J)$. Therefore $N$ is normal in $M$. In order to
show that $N$ is
normal in $I$ it suffices to prove for every $n$ in $N$, $j$ in $J$,
and $m$ in $M$ that $^{jm}n$ is in $N$. Since $N$
is normal in $M$, $^mn$ is an element of $N$. Since $J$
commutes with $N$ it follows that $^{jm}n=\,^j(^mn)=\,^mn$. It is now clear
that $^{jm}n$ is in $N$, and therefore (1) is proven.

The homomorphism $\gamma$ is well defined because $N$ commutes with $J$. It
remains to show that $\gamma$ is injective. Indeed if
$\gamma(mN)=id$ then for every $j\in J$, we have
$^mj=j$. Therefore $m$ commutes with $J$. Since
$m$ is also in $M$, we conclude that it is in $N$. Therefore
$mN=N$. This proves (2).

It is now an easy verification that the map $I=J\rtimes M\rightarrow
J\rtimes_{\gamma}(M/N)$ taking $jm$, where $j\in J$ and $m\in M$, to $(j,mN)$ is a well-defined homomorphism with kernel $N$, proving (3).
\end{proof}

We are now ready for the main theorem of this section:

\begin{thm} \label{E:tech}
Let $G$ be a finite group, and let $X_{\bar \q}\rightarrow \p_{\bar \q}$ be a $G$-Galois
branched cover that descends as a mere cover to a number field $L$. Let
$X_L\rightarrow
\p_L$ be a model of it over $L$, and let $\mathcal{A}$ be the set of all
overfields $E$ of $L$ such that $X_L\times_L E \rightarrow \p_E$ is Galois. Then there is a field $E$ in $\mathcal{A}$ that is
contained in all of the other fields in $\mathcal{A}$,
and it satisfies the following properties:
\begin{enumerate}
\item The field extension $E/L$ is Galois, with Galois group isomorphic to a
subgroup $H$ of $\oper{Aut}(G)$.
\item For every $G$-Galois field extension $F/E$ coming from specializing the
$G$-Galois branched cover $X_L\times_L E\rightarrow \p_E$ at an $E$-rational point, the field extension $F/L$ is Galois
with Galois group isomorphic to $G\rtimes H$ (where $\oper{Gal}(F/E)\cong G$ is the
obvious subgroup of $G\rtimes H$, and where the action of $H$ on $G$ is
given by the embedding of $H$ in $\oper{Aut}(G)$).
\end{enumerate}
\end{thm}

\begin{proof}

Let $L(x)$ be the function field of $\p_L$, where $x$ is a transcendental
element. By Lemma 2.4 in
\cite{syb2}, $\kappa(X_{\bar \q})$ is Galois over $L(x)$. Let $s:\oper{Gal}(L)\rightarrow \oper{Gal}(\kappa(X_{\bar
\q})/L(x))$ be the section corresponding to $X_L\rightarrow \p_L$ via the bijection $\alpha$ from Lemma
\ref{E:modelssections}. 

Let $V$ be the intersection of
$s(\oper{Gal}(L))$ with
the centralizer of $G$ in $\oper{Gal}(\kappa(X_{\bar \q})/L(x))$. 
Applying Lemma \ref{E:group} with $G$ in the role of $J$, $s(\oper{Gal}(L))$
in the role of $M$, $V$ in the role of $N$, and
$\oper{Gal}(\kappa(X_{\bar \q})/L(x))$ in the role of $I$, we see that $V$ is normal in $\oper{Gal}(\kappa(X_{\bar \q})/L(x))$, and that
$\oper{Gal}(\kappa(X_{\bar \q})/L(x))/V$ is isomorphic to a semi direct product of $G$ with a subgroup of $\oper{Aut}(G)$. In particular,
the
group $V$ is has finite index in 
$\oper{Gal}(\kappa(X_{\bar \q})/L(x))$, and therefore the compositum $GV$ is an
open subgroup of $\oper{Gal}(\kappa(X_{\bar \q})/L(x))$ containing $G$. Therefore there exists a finite field extension $E$ of $L$,
contained in
$\bar \q$, such
that the
fixed
subfield of $\kappa(X_{\bar \q})$ by $GV$ is equal to $E(x)$. Note that
$\kappa(X_L\times_LE)$ is
the fixed subfield of $\kappa(X_{\bar \q})$ by $V$.

We first show that $E$ is an element of $\mathcal{A}$, and in fact the least
element (i.e. $\forall E'\in \mathcal{A}\,\,\,E\subseteq E'$). By Lemma
\ref{E:modelssections} and Remark \ref{E:functorialconstruction}, the map $X_L\times_LE\rightarrow \p_E$ is Galois because
the image of the restriction of $s$ to $\oper{Gal}(E)$ commutes
with $G$. If $E'$ is another element of $\mathcal{A}$, then again by Lemma
\ref{E:modelssections} and Remark \ref{E:functorialconstruction} the image of the restriction of $s$
to $\oper{Gal}(E')$ commutes with $G$. But this implies that $\oper{Gal}(\kappa(X_{\bar
L})/E'(x))$ is contained in $GV$. This proves that $E$ is the least element in
$\mathcal{A}$.

The group $\oper{Gal}(E/L)\cong\oper{Gal}(E(x)/L(x))$ is isomorphic $s(\oper{Gal}(L))/V$ by the second isomorphism theorem. It follows from
the above that
$\oper{Gal}(E/L)$ embeds into $\oper{Aut}(G)$. This proves Claim (1) of Theorem \ref{E:tech}. 

Claim (3) of Lemma \ref{E:group}, applied to
our situation as above, implies that
the field extension $\kappa(X_L\times_LE)/L$ is Galois
with Galois group isomorphic to $G\rtimes H$ (where the action of $H$ on $G$ is
given by the embedding of $H$ in $\oper{Aut}(G)$); and that furthermore, we have
$\kappa(X_L\times_LE)^G=E(x)$. Claim (2) in Theorem \ref{E:tech} is now proven by
specializing.

\end{proof}

\begin{rmk}\label{E:explanation}
\rm
Let $X_{\bar \q} \rightarrow \p_{\bar \q}$ be a $G$-Galois branched cover with field of moduli $M$. Recall that the field $M$ is the
intersection of all of the fields of definition of $X_{\bar \q} \rightarrow \p_{\bar \q}$ as a $G$-Galois branched cover, but is not
necessarily one itself. However, since $M$ contains the field of moduli
of $X_{\bar \q} \rightarrow \p_{\bar \q}$ as a mere cover, it is a field of definition as a mere cover. (See Section \ref{E:introduction}.)
In light of Theorem \ref{E:tech}, one can explain the failure of $M$ to be a field of definition as a $G$-Galois branched cover as the
combination of two factors:
\begin{enumerate}
 \item Theorem \ref{E:tech} gives a unique minimal field of definition as a $G$-Galois branched cover for any particular mere cover model
$X_M\rightarrow \p_M$. However
each model might give a different minimal field of definition. Therefore the non-uniqueness of a model of $X_{\bar \q} \rightarrow \p_{\bar
\q}$ over
$M$ contributes to the plurality of the minimal fields of definition.
\item If $L$ is an overfield of $M$, then there may be a mere cover model of $X_{\bar \q}\rightarrow \p_{\bar \q}$ over
$L$ that does not descend to a mere cover model over $M$.

\end{enumerate}

\end{rmk}

This theorem has a number of noteworthy corollaries. 

\begin{crl}\label{E:thebigone} Let $G$ be a finite group, and let $X_{\bar \q}\rightarrow \p_{\bar \q}$ be a $G$-Galois branched cover. Then
the following hold:
\begin{enumerate}
 \item Assume $X_{\bar \q}\rightarrow \p_{\bar \q}$ descends as a mere cover to a number field $L$ (i.e., $L$ contains the
field of moduli as a mere cover). Then there exists a field of definition for $X_{\bar \q}\rightarrow \p_{\bar \q}$ as a $G$-Galois branched
cover that is Galois over $L$ with Galois group a subgroup of $\oper{Aut}(G)$. In particular this holds when $L$ is the field of moduli of
$X_{\bar \q}\rightarrow \p_{\bar \q}$ as a $G$-Galois branched cover.
\item Assume $X_{\bar \q}\rightarrow \p_{\bar \q}$ descends as a mere cover to a number field $L$. Then there exists a subgroup  $H\leq
\oper{Aut}(G)$ such that $G\rtimes H$ is realizable as a Galois group over $L$.
\item Let $F$ be the field of
moduli of
$X_{\bar \q}\rightarrow \p_{\bar \q}$ as a mere cover, and let $M$ be the field of moduli of $X_{\bar \q}\rightarrow \p_{\bar \q}$ as a
$G$-Galois branched cover. Then $M$ is Galois over $F$ with Galois group a subquotient of $\oper{Aut}(G)$.

\end{enumerate}

\end{crl}
\begin{proof}
 Claims (1) and (2) follow immediately from Theorem \ref{E:tech}. In light of Theorem \ref{E:tech}, in order to prove Claim (3) it suffices
to show that $M$ is Galois over $F$. Recall that $M$ is the intersection of all of the fields of definition as a $G$-Galois branched cover.
It therefore suffices to prove that for every field of definition $L$ of $X_{\bar \q}\rightarrow \p_{\bar \q}$ as a $G$-Galois branched
cover, and for every $\sigma$ in $\oper{Gal}(\bar \q/F)$, the field $\sigma L$ is also a field of definition as a $G$-Galois branched cover.
Let $X_L\rightarrow \p_L$ be an $L$-model as a $G$-Galois branched cover, and let $X_{\sigma L}\rightarrow \p_{\sigma L}$ be its twist by
$\sigma$. This cover is clearly Galois. Furthermore, note that $X_{\sigma L}\rightarrow \p_{\sigma L}$ is a mere cover model over $\sigma
L$ of the cover
$X_{\bar \q}\rightarrow \p_{\bar \q}$ after it has been twisted by $\sigma$. By the definition of $F$, the cover resulting from twisting
$X_{\bar \q}\rightarrow \p_{\bar \q}$ by $\sigma$ is isomorphic to $X_{\bar \q}\rightarrow \p_{\bar \q}$ as a mere cover. Therefore $\sigma
L$ is a field of definition of $X_{\bar \q}\rightarrow \p_{\bar \q}$ as a mere cover, and $X_{\sigma L}\rightarrow \p_{\sigma L}$ is a mere
cover model of this cover that is Galois. In other words, the field $\sigma L$ is field of definition of $X_{\bar \q}\rightarrow \p_{\bar
\q}$ as a $G$-Galois branched cover, which is what we wanted to prove.
\end{proof}

\begin{rmk}\rm Note that Claim (3) in Corollary \ref{E:thebigone} implies that there exists a subgroup $H\leq \oper{Aut}(G)$
such that $G\rtimes H$ is a Galois group over $L$ without proving it is
realizable regularly (i.e.\ as the Galois group of a regular extension of
$L(x)$).
\end{rmk}

\section{Adjoining Roots of Unity to a Field of Moduli to get a Field of Definition}\label{E:adjoin}
While Theorem \ref{E:tech} describes a general relationship between the field of
moduli and fields of definition, the main theorem of this section (Theorem
\ref{E:abelian}) describes the existence of a particular field of definition (infinite over the field of moduli)
with special properties.

Let $G$ be a finite group, and let $X_{\bar \q}\rightarrow \p_{\bar \q}$ be a $G$-Galois branched cover. As noted in Section
\ref{E:introduction}, its field of moduli $M$ as a $G$-Galois branched cover may not be a field of definition as a $G$-Galois branched
cover. However, Coombes
and Harbater
(\cite{ch})
have proven that the field $\cup_nM(\zeta_n)$ resulting from adjoining all of the roots of unity to $M$ \it is \rm a field of definition.
(Here $\zeta_n$ is defined to be $e^{\frac{2\pi i}{n}}$.) The following is a
strengthening of this result.
\begin{thm} \label{E:abelian}
In the situation above, the field $\cup_{\{n|\exists
m:\,n|\,|Z(G)|^m\}}M(\zeta_n)$ is a field of definition. In particular, there
exists a field of definition (finite over $\q$) that is ramified over the field of moduli $M$ only
over the primes that divide $|Z(G)|$.
\end{thm}
\begin{proof}
If $G$ is centerless, then the cover is defined over its field of moduli
(\cite{ch}) and therefore the theorem follows. Otherwise $\cup_{\{n|\exists
m:\,n|\,|Z(G)|^m\}}M(\zeta_n)$ satisfies the
hypotheses of Proposition 9 in Chapter II of \cite{serregalois}. We conclude
that $\oper{cd}_p(\cup_{\{n|\exists
m:\,n|\,|Z(G)|^m\}}M(\zeta_n))\leq 1$ for every prime $p$ that divides
$|Z(G)|$. This implies that $H^2(\cup_{\{n|\exists
m:\,n|\,|Z(G)|^m\}}M(\zeta_n),Z(G))$ is trivial. As the obstruction for this
field to be a field of definition lies in this group (\cite{descent}), we are
done.
\end{proof}
Combining Theorem \ref{E:abelian} with results that I have proven in \cite{me}, we get the following corollaries.
\begin{crl} \label{E:coolerigp} Let $G$ be a finite group. Then the following hold:
\begin{enumerate}
 \item
For every positive integer $r$ there is a set $T=\{a_1,...,a_r\}$ of closed
points of $\p_{\bar \q}$, such that every $G$-Galois branched cover of $\p_{\bar
\q}$ that is ramified only over $T$, has a field of definition that is unramified (over $\q$) outside of the primes dividing $|G|$.

\item For every positive integer $r$, and for every finite set $S$ of rational primes that don't divide
$|G|$, there is a choice of $\q$-rational points
$T=\{a_1,...,a_r\}$ such that every $G$-Galois \'etale cover of $\p_{\bar
\q}\smallsetminus T$ has a field of definition that is unramified (over $\q$) over 
the primes of $S$.

\item There is an extension of number
fields $\q\subset E\subset F$ such that $F/E$ is $G$-Galois, and $E/\q$
ramifies only over those primes that divide $|G|$.

\end{enumerate}
\end{crl}
\begin{proof}
Claims (1) and (2) of the corollary are straightforward from Theorems 9.1 and 9.6 of \cite{me}
respectively, together with Theorem \ref{E:abelian}. Claim (3) follows from Claim (1) by specializing.
\end{proof}

\medskip
\noindent Current author information:\\
Hilaf Hasson: Department of Mathematics, Pennsylvania State University, State College, PA 16802, USA\\
email: {\tt hilafhasson@gmail.com}
\end{document}